\documentclass[12pt]{article}

\usepackage{diagbox}
\usepackage{mathtools}
\usepackage{bbm}
\usepackage{latexsym}
\usepackage{epsfig}
\usepackage{amsmath,amsthm,amssymb,enumerate}
\usepackage{breqn}

\usepackage[a-1b]{pdfx}
\usepackage{hyperref}

\usepackage{color}

   \def\D{\Delta}
\def\e{\epsilon}

 \def\om{\omega}

\def\A{\forall}
\newtheorem{theorem}{Theorem}
\newtheorem{lemma}[theorem]{Lemma}
\newtheorem{corollary}[theorem]{Corollary}
\newtheorem{Remark}{Remark}

\newcommand{\brac}[1]{\left(#1\right)}

\newcommand{\bfrac}[2]{\left(\frac{#1}{#2}\right)}

\newcommand{\ra}{\rightarrow}

\allowdisplaybreaks
\newcommand{\set}[1]{\left\{#1\right\}}

\def\E{\mathbb{E}}

\def\Pr{\mathbb{P}}

\newcommand{\ignore}[1]{}

\def\cL{{\mathcal L}}

\def\cO{{\mathcal O}}

\newcommand{\beq}[2]{\begin{equation}\label{#1}#2\end{equation}}

\newcommand{\kk}[2]{k^{(#1)}_{#2}}

\usepackage{tikz}
\usetikzlibrary{decorations.pathmorphing}
\usetikzlibrary{positioning}
\usetikzlibrary{arrows,automata}
\usetikzlibrary{shapes.misc}
\usetikzlibrary{backgrounds}
\usetikzlibrary{arrows,shapes}

\def\N{\mathbb{N}}

\begin{document}
	
\author{Alan Frieze\thanks{Research supported in part by NSF grant DMS1952285} \ and Aditya Raut\\Department of Mathematical Sciences\\Carnegie Mellon University\\Pittsburgh PA-15213}
\date{}
\title{The maximum degree of the $r$th power of a sparse random graph}
\maketitle

\begin{abstract}
Let $G^r_{n,p}$ denote the $r$th power of the random graph $G_{n,p}$, where $p=c/n$ for a positive constant $c$. We prove that w.h.p. the maximum degree $\D\left(G^r_{n,p}\right)\sim \frac{\log n}{\log_{(r+1)}n}$. Here $\log_{(k)}n$ indicates the repeated application of the log-function $k$ times. So, for example, $\log_{(3)}n=\log\log\log n$. 
\end{abstract}

\section{Introduction}

The $r$th power $G^r$ of a graph $G$ is obtained from $G$ by adding edges for all pairs of vertices at distance $r$ or less from each other. Powers of graphs arise naturally in various contexts, e.g. in the study of Shannon capacity. In the context of random graphs, there has been little research. 

Let $\log_{(k)} n$ indicate the repeated application of the log-function $k$ times.  So, for example, $\log_{(3)} n = \log\log\log n$. 

One basic property of a class of graphs is their degree sequence. In particular, the maximum degree is of particular interest. Garapaty, Lokshtanov, Maji and Pothen \cite{GLMP} proved that if $p=c/n$ where $c>0$ is constant, then w.h.p. the maximum degree $\D(G_{n,p}^r)=\Theta_r\left(\frac{\log n}{\log_{(r+1)} n}\right)$. The hidden multiplicative factor being in the range $[0.05\cdot 2^{-r}, 6]$. We strengthen this and prove

\begin{theorem}\label{th1}
Let $p=c/n$, where $c>0$ is a constant and let $r\geq 2$ be a fixed positive integer. Then, w.h.p. $\D(G_{n,p}^r) \sim\frac{\log n}{\log_{(r+1)} n}$.
\end{theorem}
(The case $r=1$ is well-known, see e.g. Theorem 3.4 of \cite{FK}.)
\begin{Remark}
The value of $c$ does not contribute to the main term in the claim of Theorem \ref{th1}. Thus we would expect that we could replace $p=c/n$ by $p\leq \om/n$ for some slowly growing function $\om=\om(n)\to \infty$.
\end{Remark}
\section{Proof of Theorem \ref{th1}}
Given $v\in[n]=V(G), G=G_{n,p}$, let $N_t(v)$ denote the set of vertices of $G$ at distance $t$ from $v$. Let $d_t(v) := |N_t(v)|$. The probability that $|N_t(v)|=\ell_t,t=1,2,\ldots,r$ can be calculated as follows. Let the $\ell_i$ neighbors of $v$ at distance $i$ have $\kk{i}{1}, \kk{i}{2}, ...,\kk{i}{\ell_i}$ neighbors at distance $i+1$ respectively. Then, 
\begin{dmath*}
	\Pr\left[d_i(v)=\ell_i \ \A i =1,2,...,r \right] = \binom{n-1}{\ell_1} p^{\ell_1} (1-p)^{n-1-\ell_1} \left(\sum_{\substack{\kk{1}{1}, \kk{1}{2}, ..., \kk{1}{\ell_1} \\ \kk{1}{1}+ ... + \kk{1}{\ell_1} = \ell_2}} \prod_{i=1}^{\ell_1} \binom{n-(1+\ell_1)}{\kk{1}{i}} p^{\kk{1}{i}} (1-p)^{n-(1+\ell_1)-\kk{1}{i}} \times \cdots \\
	\cdots \left( \sum_{\substack{\kk{r-1}{1}, \kk{r-1}{2}, ..., \kk{r-1}{\ell_{r-1}} \\ \kk{r-1}{1}+ ... + \kk{r-1}{\ell_{r-1}} = \ell_r}} \prod_{i=1}^{\ell_{r-1}} \binom{n-(1+\ell_1+...+\ell_{r-1})}{\kk{r-1}{i}} p^{\kk{r-1}{i}} (1-p)^{n-(1+\ell_1 +...+\ell_{r-1})-\kk{r-1}{i}} \right)\right)
\end{dmath*}
Now it is well-known that $\D(G_{n,p})=o(\log n)$ w.h.p., see for example \cite{FK}. In consequence $\ell_i=o(\log^in)$ for $i\leq r$.  We use the approximation $\binom{n}{\ell}= \frac{n^\ell}{\ell !}\cdot\brac{1+{\cO}\bfrac{\ell^2}{n}}$ and simplify 
$$\binom{n}{\ell} p^\ell (1-p)^{n-\ell} =\frac{n^\ell}{\ell !}\cdot \frac{c^\ell}{n^\ell} \cdot \left(1-\frac{c}{n}\right)^{n-\ell}\cdot\brac{1+{\cO}\bfrac{\ell^2}{n}} = \frac{c^\ell e^{-c}}{\ell!}\cdot\brac{1+{\cO}\bfrac{\ell^2}{n}} $$
Simplifying all terms this way, we have
\begin{dmath*} 
	\begin{aligned}
		& \Pr\left[d_i(v)=\ell_i \ \A i =1,2,...,r\right] \\
		& = \frac{c^{\ell_1} e^{-c}}{\ell_1!} \left( \sum_{\substack{\kk{1}{1}, \kk{1}{2}, ..., \kk{1}{\ell_1} \\ \kk{1}{1}+ ... + \kk{1}{\ell_1} = \ell_2}} \prod_{i=1}^{\ell_1} \frac{c^{\kk{1}{_i}} e^{-c}}{\kk{1}{i}!} \times  \cdots \left( \sum_{\substack{\kk{r-1}1, \kk{r-1}2, ..., \kk{r-1}{\ell_{r-1}} \\ \kk{r-1}1+ ... + \kk{r-1}{\ell_{r-1}} = \ell_r}} \prod_{i=1}^{\ell_{r-1}} \frac{c^{\kk{r-1}i}e^{-c}}{\kk{r-1}i!} \right) \right) \cdot \left(1+\cO\left(\frac{\sum\limits_{i=1}^r\ell_i^2}{n}\right)\right) \\
		& = \frac{c^{\ell_1} e^{-c}}{\ell_1!} \left( \frac{c^{\ell_2}e^{-c\ell_1}}{\ell_2!} \sum_{\substack{\kk11, \kk12, ..., \kk1{\ell_1} \\ \kk11+ ... + \kk1{\ell_1} = \ell_2}} \frac{\ell_2!}{\prod\limits_{i=1}^{\ell_1} \kk1i!} \times  \cdots \left(\frac{c^{\ell_r}e^{-c\ell_{r-1}}}{\ell_r!} \sum_{\substack{\kk{r-1}1, \kk{r-1}2, ..., \kk{r-1}{\ell_{r-1}} \\ \kk{r-1}1+ ... + \kk{r-1}{\ell_{r-1}} = \ell_r}} \frac{\ell_r!}{\prod\limits_{i=1}^{\ell_{r-1}} \kk{r-1}i!}\right) \right) \cdot \left(1+\cO\left(\frac{\sum\limits_{i=1}^r\ell_i^2}{n}\right)\right) \\
	\end{aligned}
\end{dmath*}
We now use $\sum\limits_{\substack{k_1,...,k_t \\ k_1+...+k_t = m}} \hspace{-15pt}\binom{m}{k_1,...,k_t} = t^m$ to obtain

\begin{dmath*}
	\begin{aligned}
		\Pr\left[d_i(v)=\ell_i \ \A i =1,2,...,r\right]
		& = \frac{c^{\ell_1} e^{-c}}{\ell_1!} \cdot\left( \frac{c^{\ell_2}e^{-c\ell_1}}{\ell_2!} \cdot \ell_1^{\ell_2} \times  \cdots \left(\frac{c^{\ell_r}e^{-c\ell_{r-1}}}{\ell_r!} \cdot \ell_{r-1}^{\ell_r}\right) \right) \cdot \left(1+\cO\left(\frac{\sum\limits_{i=1}^r\ell_i^2}{n}\right)\right) \\
		& = \frac{c^{(\ell_1+...+\ell_r)} e^{-c(1+\ell_1+...+\ell_{r-1})}}{\ell_1!\ell_2!...\ell_r!} \cdot \ell_1^{\ell_2}\cdot\ell_2^{\ell_3}\cdots\ell_{r-1}^{\ell_r} \cdot \left(1+\cO\left(\frac{\sum\limits_{i=1}^r\ell_i^2}{n}\right)\right) \\
	\end{aligned} 
\end{dmath*}

The exact probability that we are interested in is the degree of $v$ being $d$ in $G^r$, i.e. $\sum\limits_{i=1}^r d_i(v) = d$.\\
We aim to show that $d \sim \frac{\log n}{\log_{(r+1)} n}$ happens with high probability. 

$$\begin{aligned}
	\Pr\left[\sum_{i=1}^r d_i(v) = d\right] 
	& = \sum_{\substack{\ell_1,...,\ell_r \\ \ell_1+...+\ell_r = d}} \frac{c^{d} e^{-c(1+d-\ell_r)}}{\ell_1!\ell_2!...\ell_r!} \cdot \ell_1^{\ell_2}\cdot\ell_2^{\ell_3}\cdots\ell_{r-1}^{\ell_r} \cdot \left(1+\cO\left(\frac{\sum\limits_{i=1}^r\ell_i^2}{n}\right)\right) \\
	& = \left(\sum_{\substack{\ell_1,...,\ell_r \\ \ell_1+...+\ell_r = d}} u_{\ell_1,...,\ell_r}\right) \left(1+\cO\left(\frac{d^2}{n}\right)\right) ,
\end{aligned}$$
where 
\beq{ul}{
u_{\ell_1,...,\ell_r} = \frac{c^de^{-c(1+d-\ell_r)}}{\ell_1!\ell_2!...\ell_r!}\cdot \ell_1^{\ell_2}\cdots \ell_{r-1}^{\ell_r}.
}

For completeness, define $\ell_0=1$. Using Stirling's approximation $\log(n!) = n\log n - n + \cO(\log n)$, we have 
\beq{bound}{
\log u_{\ell_1,...,\ell_r} = d\log c - c(1+d-\ell_r) - \sum\limits_{i=1}^r \ell_i \log\frac{\ell_i}{\ell_{i-1}} +\cO(d)
}
The following lemma bounds the sum in \eqref{bound}: 	
\begin{lemma}\label{lem2}
	For $\ell_0=1$ and $\ell_1,...,\ell_r \in \N$ such that $\sum\limits_{i=1}^r \ell_i = d$, we have $\min \sum\limits_{i=1}^r \ell_i \log \frac{\ell_i}{\ell_{i-1}} \geq d\log_{(r)} d + \cO(d)$, for sufficiently large $d$.
\end{lemma}
\begin{proof}
	We proceed by induction on $r$. For $r=1$, the result holds since we have $\ell_1=d$, implying that $\ell_1\log\frac{\ell_1}{\ell_0} = d\log_{(1)}d$. Assume that the result holds for $r-1$.  

{\bf Case 1: $(d-\ell_r)\log_{(r-1)}(d-\ell_r) \geq d\log_{(r)} d$:}\\
Because $\sum\limits_{i=1}^{r-1} \ell_i = d-\ell_r$, from the induction hypothesis we have $\sum\limits_{i=1}^{r-1} \ell_i\log\frac{\ell_i}{\ell_{i-1}} \geq (d-\ell_r)\log_{(r-1)}(d-\ell_r)$. So this case is done.

{\bf Case 2: $(d-\ell_r)\log_{(r-1)}(d-\ell_r)< d\log_{(r)} d$:}\\
For $d$ sufficiently large, we have $\frac{d}{10}\log_{(r-1)} \frac{d}{10} > d\log_{(r)} d$. Hence $\ell_r \geq \frac{9}{10}\cdot d$, implying $\ell_{r-1} \leq \frac{d}{10}$ and $\frac{\ell_r}{\ell_{r-1}} \geq 9$.
	
We now use the method of Lagrange multipliers. But first we deal with the constraints $\ell_i\geq 0$ for $i=1,2,\ldots,r$. If $\ell_i=0$ and $\ell_{i+1}\neq 0$ then $\sum\limits_{i=1}^r \ell_i \log \frac{\ell_i}{\ell_{i-1}}=\infty$. If $\ell_i=\ell_{i+1}=\cdots=\ell_r=0$ then $\sum\limits_{i=1}^r \ell_i \log \frac{\ell_i}{\ell_{i-1}}=\sum\limits_{i=1}^{i-1} \ell_i \log \frac{\ell_i}{\ell_{i-1}}$ and the result follows by induction.  So, in effect there is one constraint: $\sum\limits_{i=1}^r \ell_i = d$.

Define $\cL(\ell_1,...,\ell_r,\lambda) = \sum\limits_{i=1}^r \ell_i \log \frac{\ell_i}{\ell_{i-1}} + \lambda \left(\sum\limits_{i=1}^r \ell_i - d\right)$. By the Lagrange multiplier theorem, the minima must satisfy $\frac{\partial\cL}{\partial \ell_i} = 0$ for all $i$. Notice that $\frac{\partial \cL}{\partial \ell_i} = 1 + \log\frac{\ell_{i}}{\ell_{i-1}} - \frac{\ell_{i+1}}{\ell_i} + \lambda$ for $1 \leq i < r$, and $\frac{\partial\cL}{\partial\ell_r} = 1+\log \frac{\ell_r}{\ell_{r-1}} +\lambda$. Let $p_i := \frac{\ell_i}{\ell_{i-1}}$. From $\frac{\partial\cL}{\partial \ell_r}=0$, we have $1+\log p_r + \lambda = 0$ which implies that $p_r = e^{-(\lambda+1)}$. For $1\leq i<r$, from $\frac{\partial\cL}{\partial\ell_i}=0$ we have $1+\log p_i -p_{i+1}+\lambda = 0$ which  implies that $p_i = e^{p_{i+1}-(1+\lambda)} = p_r\cdot e^{p_{i+1}}\geq p_rp_{i+1}$. Thus, we can iteratively obtain the exact expressions for $p_1,p_2,...,p_{r-1}$ in terms of $p_r$. Now recall that $p_r\geq 9$ from induction hypothesis, hence $p_i \geq 9$ for all $i$. 
	
Since $\ell_0=1$, $\ell_1=p_1$. Now $\ell_i = p_i\cdot p_{i-1}\cdots p_1$, for all $i=1,2,...,r$ and then $\frac{\ell_r}{\ell_i} = p_rp_{r-1}...p_{i+1} \geq 9^{r-i}$. Now $\sum\limits_{i=1}^r \ell_i = d$ and so clearly $d > \ell_r$. Moreover, $d = \ell_r \left(\sum\limits_{i=1}^r \frac{\ell_i}{\ell_r}\right) \leq \ell_r \left(\sum\limits_{i=1}^r \frac{1}{9^{r-i}}\right) < \ell_r\left(\sum\limits_{j=0}^\infty \frac{1}{9^j}\right) = \frac{9}{8}\cdot \ell_r$. We can thus assume that $d = c_1 \cdot \ell_r$ for some $c_1 \in (1,9/8)$. So,
	$$\begin{aligned}
		d & = c_1 \cdot p_r \cdot p_{r-1} \cdots p_1 \\
		& = c_1 \cdot p_r \cdot \left(p_r e^{p_r}\right) \cdot\left(p_r e^{p_r e^{p_r}}\right) \cdots \left(p_r \underbrace{e^{p_r e^{... e^{p_r}}}}_{\substack{\text{exponential tower}\\\text{of height }r-1}}\right)\\
		& = c_1 \cdot p_r^r \cdot \left(e^{p_r} e^{p_r e^{p_r}} \cdots \underbrace{e^{p_r e^{... e^{p_r}}}}_{\substack{\text{exponential tower}\\\text{of height }r-1}} \right)\\
	\end{aligned}$$
Applying the logarithmic function to both sides $(r-1)$ times, we have $p_r\leq 
\log_{(r-1)} d = p_r + \cO(\log p_r)$, implying that $\log_{(r-1)}d-c\log_{(r)}d\leq p_r\leq \log_{(r-1)}d$ for some constant $c>0$. We see from the above that $p_{r-1}=p_re^{p_r}$ and that $p_i\geq p_rp_{i+1}$. It follows that $\ell_i\leq \ell_rp_r^{i-r}$ and so 
\[
\ell_r = d(1-\eta)\text{ for }\eta \leq \frac{2}{\log_{(r-1)}d} \ra 0.
\]
 Let $T_i : = \ell_i\log\frac{\ell_i}{\ell_{i-1}}$. Then we have 
\[
\frac{T_i}{T_{i-1}}=p_i\frac{\log p_i}{\log p_{i-1}}=\frac{p_i\log p_i}{p_i+\log p_r}\geq \tfrac12\log p_i\gg1.
\]
It follows that for all $i<r-1$, we have $T_i <T_{r-1}$. Now $\log p_{r-1} = p_r + \log p_r$ implies that 
\[
T_{r-1} = \ell_{r-1}\log p_{r-1}=\ell_{r-1}(p_r + \log p_r)\leq \frac{\ell_r}{p_r}(p_r + \log p_r)  = \cO(\ell_r) = \cO(d).
\]
Thus, the objective is dominated by last summand $T_r = \ell_r\log\frac{\ell_r}{\ell_{r-1}}$, resulting in minimum value of at least $d\log_{(r)}d+\cO(d)$.
 \end{proof}

Let 
\[
d_* = \frac{\log n}{\log_{(r+1)} n}.
\]
\subsection{Upper bound on $\D(G^r)$}\label{sec3.1}
We prove that $\D(G^r) \leq d_*(1+\e)$ w.h.p., where $\e=\frac{1}{\log_{(r+1)}n}$.  The following inequality will be useful.
\begin{lemma}\label{log()}
Suppose that $a\gg b$. Then
\[
\log_{(s)}(a-b)\geq \log_{(s)}a-\frac{2sb}{a\log a\log\log a\cdots\log_{s-1}a}.
\]
\end{lemma}
\begin{proof}
We prove this by induction on $s$. For $s=1$ we have
\[
\log(a-b)=\log a+\log\brac{1-\frac{b}{a}}\geq \log a-\frac{2b}{a}.
\]
Then for $s>1$ we have
\begin{align*}
\log_{(s)}(a-b)&=\log\brac{\log_{(s-1)}(a-b)}\\
&\geq \log\brac{\log_{(s-1)}a-\frac{2(s-1)b}{a\log a\log\log a\cdots\log_{s-2}a}}\\
&\geq \log_{(s)}a-\frac{2sb}{a\log a\log\log a\cdots\log_{s-1}a}.
\end{align*}
\end{proof}
\begin{corollary}\label{cor1}
Suppose that $\log a\gg \log b$. Then
\[
\log_{(s)}(a/b)\geq \log_{(s)}a-\frac{2(s-1)\log b}{\log a\log\log a\cdots\log_{s-1}a}.
\]
\end{corollary}
\qed

Using $u_{\ell_1,...,\ell_r}$ as in \eqref{bound} and Lemma \ref{lem2},
\begin{align*}
	\log u_{\ell_1,...,\ell_r} 
	&= d \log c - c(1+d-\ell_r) - \sum\limits_{i=1}^r \ell_i \log\frac{\ell_i}{\ell_{i-1}} +\cO(d) \\
	& \leq -d_*(1+\e) \log_{(r)} (d_*(1+\e))+{\cO}(d_*)\\
	& \leq -d_*(1+\e) \log_{(r)} d_*+{\cO}(d_*)\\
	& \leq -\brac{1+\e}\log n\brac{1-O\bfrac{\log_{(r)}n}{\log n}}\\
	&\leq -\brac{1+\frac{\e}2}\log n
\end{align*}

Hence $u_{\ell_1,...,\ell_r} \leq \frac{1}{n^{1+\e/2}}$. Since there are at most $d^r$ terms in the summation, we have
$$\Pr\left[\sum\limits_{i=1}^r d_i(v) = d\right] = \left(\sum_{\substack{\ell_1,...,\ell_r \\ \ell_1+...+\ell_r = d}} u_{\ell_1,...,\ell_r}\right) \left(1+\cO\left(\frac{d^2}{n}\right)\right) \lesssim \frac{d^r}{n^{1+\e/2}}.$$

Finally by taking the union bound over all $n$ vertices, 
$$\Pr[\D_r(G) \geq d] \leq \sum_{i=1}^n \Pr\left[ \sum\limits_{j=1}^r d_j(v_i) = d\right] \lesssim n \cdot \frac{d^r}{n^{1+\e/2}} = \frac{d^r}{n^{\e/2}} \ra 0.$$

\subsection{Lower bound on $\D(G^r)$}
We now use the second moment method to show that $\D_r(G) \geq d_*(1-\e)$ w.h.p. 
For $i=1,2,...,n$, let $X_i$ be the indicator random variable for $\sum\limits_{j=1}^r d_j(v_i) > d_{*}(1-\e)$, i.e. the event that $v_i\in V(G)$ has degree greater than $d_*(1-\e)$ in $G^r$. Let $X_* = \sum\limits_{i=1}^n X_i$ and $\ell_*$ be values of $\ell_i$ which achieve the lower bound in Lemma \ref{lem2}. 
\begin{align*} 
	\Pr[X_i = 1] \geq u_{\ell_*} 	& = \exp(-d_*(1-\e)\log_{(r)} d_*(1-\e)+ \cO(d_*)) \\
	& \geq \exp(-(1-\e)\log n + \cO(d_*)) \\
	& \geq \frac{1}{n^{1-\e/2}}.
\end{align*}
Then we have 
\[
\E[X_*]  \geq n^{\e/2}. 
\]
On the other hand,
$$
\begin{aligned} \Pr[X_* > 0] & \geq \frac{\E[X_*]^2}{\E[X_*^2]} \\ 
	& = \frac{\E[X_*]^2}{\displaystyle \sum_{i=1}^n \E\left[X_i^2\right] + \sum_{i \neq j} \E\left[X_iX_j\right]} \\
	& \geq \frac{\E[X_*]^2}{\displaystyle \E\left[X_*\right] + \E\left[X_*\right] \sum_{j : d(v_1, v_j) > r} \E\left[X_j \right]+\E\left[|\set{j : d(v_1, v_j) \leq r}|\right]}.
\end{aligned}
$$
Here we use the fact $X_j\leq 1$ for all $j$ and that the only variables $X_j$ that are dependent on $X_1$ are for the vertices $v_j$ within a distance $r$ of $v_1$.

Now,
\[
\E\left[|\set{j : d(v_1, v_j) \leq r}|\right]\leq O(\log^{r+1}n)+n\Pr(\D(G_{n,p}\geq 10c\log n)\leq 2\log^{r+1}n.
\]
So,
\[
\Pr[X_* > 0] \geq \frac{\E[X_*]^2}{\E[X^*]+\E[X_*]^2+2\log^{r+1}n}\geq \frac{1}{n^{-\e/2}+1+2n^{-\e}\log^{r+1}n}\to 1.
\]
\section{Conclusions}
We have established the likely value of one of the key parameters related to powers of $G_{n,p},p=c/n$. It would be interesting to explore other parameters. The chromatic number of $G_{n,p}^2$ was asymptotically determined w.h.p. in Frieze and Raut \cite{FR} and the independence number w.h.p. (for large $c$) in Atkinson and Frieze \cite{AF}.

\end{document}